\documentclass[12pt]{amsart}
\usepackage{graphicx, comment}
\usepackage{luatex85}
\usepackage{amsmath}
\usepackage[normalem]{ulem}
\usepackage{amscd}
\usepackage{amsfonts}
\usepackage{subcaption}
\usepackage{amssymb}
\usepackage{color}
\usepackage{fullpage}
\usepackage{mathtools}
\usepackage[hypertexnames=false, colorlinks, citecolor=red,linkcolor=blue, urlcolor=black]{hyperref}
\usepackage{tikz-cd}
\usepackage{tikz}
\usepackage[capitalise]{cleveref}
\usepackage{enumitem}
\usepackage[all]{xy}

\numberwithin{equation}{section}

\newtheorem{theorem}{Theorem}[section]
\newtheorem{lemma}[theorem]{Lemma}

\newtheorem{proposition}[theorem]{Proposition}
\newtheorem{conjecture}[theorem]{Conjecture}

\newtheorem{corollary}[theorem]{Corollary}

\theoremstyle{definition}

\newtheorem{remark}[theorem]{Remark}

\newcommand{\U}{\operatorname{U}}
\newcommand{\bbF}{\mathbb{F}}

\newcommand{\Aut}{\operatorname{Aut}}

\newcommand{\PGL}{\operatorname{PGL}}

\newcommand{\PSU}{\operatorname{PSU}}

\newcommand{\Sp}{\operatorname{Sp}}
\newcommand{\PSp}{\operatorname{PSp}}

\newcommand{\GL}{\operatorname{GL}}

\newcommand{\trdeg}{\operatorname{trdeg}}
\newcommand{\Sym}{\operatorname{S}}
\newcommand{\Alt}{\operatorname{A}}
\newcommand{\ed}{\operatorname{ed}}

\newcommand{\Char}{\operatorname{char}}

\newcommand{\lev}{\operatorname{lev}}

\newcommand{\rd}{\operatorname{rd}}
\newcommand{\SU}{\operatorname{SU}}

\begin{document}
	
	\author{Oakley Edens}   
	\thanks{Oakley Edens was partially supported by an Undergraduate Student Research Award (USRA) from 
		the National Sciences and Engineering Research Council of Canada.} 
	
	\author{Zinovy Reichstein}
	\address{Department of Mathematics\\
		University of British Columbia\\
		Vancouver, BC V6T 1Z2\\Canada}
	\thanks{Zinovy Reichstein was partially supported by an Individual Discovery Grant from the
		National Sciences and Engineering Research Council of
		Canada.}
	
	\subjclass[2020]{12E05, 14G17}
	
	\keywords{Resolvent degree, Hilbert's 13th Problem, positive characteristic}
	\title{Hilbert's 13th problem in prime characteristic}
	
	\begin{abstract}
		The resolvent degree $\rd_{\mathbb C}(n)$ is the smallest integer $d$ such that a root of the general polynomial
		\[ f(x) = x^n + a_1 x^{n-1} + \ldots + a_n \]
		can be expressed as a composition of algebraic functions in at most $d$ variables with complex coefficients. 
		It is known that $\rd_{\mathbb C}(n) = 1$ when $n \leqslant 5$. Hilbert was particularly interested in the next three cases: he asked if $\rd_{\mathbb C}(6) = 2$ (Hilbert's Sextic Conjecture), $\rd_{\mathbb C}(7) = 3$ (Hilbert's 13th Problem) and $\rd_{\mathbb C}(8) = 4$ (Hilbert's Octic Conjecture). These problems remain open. It is known that $\rd_{\mathbb C}(6) \leqslant 2$,
		$\rd_{\mathbb C}(7) \leqslant 3$ and $\rd_{\mathbb C}(8) \leqslant 4$. It is not known whether or not $\rd_{\mathbb C}(n)$ can be $> 1$ for any
		$n \geqslant 6$. 
		
		In this paper, we show that all three of Hilbert's conjectures can fail if we replace $\mathbb C$ with a base field of positive characteristic.
	\end{abstract}
 
	\maketitle

     \section{Introduction}
	The algebraic form of Hilbert's 13th Problem asks for the resolvent degree $\rd_{\mathbb C}(n)$ of the general polynomial 
	\begin{equation*}
		f(x)=x^n+a_1x^{n-1}+\ldots+ a_{n-1}x + a_n,
	\end{equation*}
	where $a_1, \ldots , a_n$ are independent variables. Here $\rd_{\mathbb C}(n)$ is the minimal integer $d$ such that every root of $f(x)$ can be obtained in a finite number of steps, starting with $\mathbb C(a_1, ..., a_n)$ and adjoining an algebraic function in $\leqslant d$ variables at each step. For a precise definition, see~\cite{brauer}, \cite{farb-wolfson},  \cite{reichstein2022hilberts}, \cite{reichstein-ems} or Section~\ref{sect.prel} below.
	It is known that $\rd_{\mathbb C}(n) = 1$ for every $n \leqslant 5$. It is not known whether or not $\rd_{\mathbb C}(n)$ is bounded from above, as $n$ tends to infinity or even 
	if $\rd_{\mathbb C}(n)$ can be greater than $1$ for any $n$. Various upper bounds on $\rd_{\mathbb C}(n)$ have been proved 
	over the past 200 years. For an overview, see \cite{dixmier}. These classical bounds have recently been sharpened by Wolfson~\cite{wolfson}, Sutherland~\cite{sutherland}, and Heberle-Sutherland~\cite{heberle-sutherland}.
	All of them are of the form $\rd_{\mathbb C}(n)\leqslant n-\alpha(n)$, where $\alpha(n)$ 
	is an unbounded but very slow-growing function of $n$. 
	There is a wide gap between the best known lower bound, $\rd_{\mathbb C}(n) \geqslant 1$, and the best known upper bound, 
	$\rd_{\mathbb C}(n) \leqslant n - \alpha(n)$. It is fair to say that after two 
	centuries of research, we still know very little 
	about $\rd_{\mathbb C}(n)$ for $n \geqslant 6$. Specifically, Hilbert conjectured the following values for small $n$.
	
	\begin{conjecture} \label{conj.hilbert} 
		{\rm (a)} $\rd_{\mathbb C}(6) = 2$, {\rm (b)} $\rd_{\mathbb C}(7)=3$, {\rm (c)} $\rd_{\mathbb C}(8) = 4$.
	\end{conjecture}
	
	(a) and (c) appeared in~\cite[p.~247]{hilbert9}; they 
		are known as Hilbert's sextic and octic conjectures, respectively.
		(b) is taken from the statement of Hilbert's 13th Problem~\cite[p.~424]{hilbert-problems}.
		The upper bounds, 
		\begin{equation} \label{e.upper-bounds}
			\text{$\rd_{\mathbb C}(6) \leqslant 2$, $\rd_{\mathbb C}(7) \leqslant 3$ and $\rd_{\mathbb C}(8) \leqslant 4$} 
		\end{equation} 
		go back to the work of Hamilton in the 1830s~\cite{hamilton}; for modern treatments, see~\cite[p.~87]{dixmier} 
		or~\cite[Corollary 7.3]{farb-wolfson}. The reverse inequalities remain out of reach.
	
	Recently Farb and Wolfson~\cite{farb-wolfson} defined the resolvent degree $\rd_k(G)$, where $G$ is a finite group and $k$ is a field 
	of characteristic $0$. Setting $G$ to be the symmetric group $\Sym_n$ and $k$ to be the field $\mathbb C$ of complex numbers, we recover
	$\rd_{\mathbb C}(n)$. This definition was extended by the second author~\cite{reichstein2022hilberts} to the case, where $k$ is an arbitrary field
	and $G$ is an arbitrary algebraic group over $k$. For a fixed algebraic group $G$ defined over the integers, $\rd_k(G_k)$ depends only on
	the characteristic of $k$ and not on $k$ itself; see~\cite[Theorem 1.2]{reichstein2022hilberts}. 
	We will write $\rd_p(G)$ in place of $\rd_k(G)$, when $k$ is a field of characteristic $p \geqslant 0$.
	Moreover, $\rd_0(G) \geqslant \rd_p(G)$ for any $p > 0$; see~\cite[Theorem 1.3]{reichstein2022hilberts}.
	
	In view of the last inequality, it is natural to ask if more can be said about Conjecture~\ref{conj.hilbert} in the case, 
	where the base field $\mathbb C$ is replaced by a field $k$ of positive characteristic. Conjecturally, one expects $\rd_p(G)$ to be the same as $\rd_0(G)$ when $p$ does not divide the order of $G$. We will thus 
	examine $\rd_p(\Sym_n)$ in the case, where $n = 6, 7, 8$ and $2 \leqslant \Char(k) = p \leqslant n$. Our main result is as follows.
	
	\begin{theorem} \label{thm.main} Let $\Sym_n$ denote the symmetric group on $n$ letters. Then
		
		\smallskip
		{\rm (a)} $\rd_3(\Sym_6) \leqslant 1$, 
		\quad 
		{\rm (b)} $\rd_3(\Sym_7) \leqslant 2$, 
		\quad
		{\rm (c)} $\rd_5(\Sym_7) = \rd_5(\Sym_6) \leqslant 2$,
		
		\smallskip
		{\rm (d)} $\rd_7(\Sym_7) \leqslant 2$, 
		\quad
		{\rm (e)}  $\rd_2(\Sym_8) \leqslant 3$.
	\end{theorem}
	
     In particular, every part of Conjecture~\ref{conj.hilbert} fails if $\mathbb{C}$ is replaced by a base field of (suitable) positive characteristic.  
     
     Theorem~\ref{thm.main} may be viewed as complementing the results of~\cite{wolfson},~\cite{sutherland},~\cite{heberle-sutherland} and~\cite{gsw}. These papers generalize the inequalities $\rd_0(\Sym_6) \leqslant 2$, $\rd_0(\Sym_7) \leqslant 3$ and
 $\rd_0(\Sym_8) \leqslant 4$ of~\eqref{e.upper-bounds} by giving upper bounds on $\rd_0(G)$, when $G$ is the symmetric group $\Sym_n$ ($n$ arbitrary)~\cite{wolfson},~\cite{sutherland},~\cite{heberle-sutherland} or when $G$ is a sporadic finite simple group~\cite{gsw}. 
 Here we stay with $G = \Sym_6$, $\Sym_7$, $\Sym_8$ and prove sharper bounds on $\rd_p(G)$ for suitable small primes $p$.  
	
	We also consider the Weil group $W(E_6)$ of the root system of type $E_6$. It is shown in \cite[Section 8]{farb-wolfson} that
	this group arises naturally in connection with Conjecture~\ref{conj.hilbert}(a), and that $\rd_0(W(E_6)) \leqslant 3$; see also
	\cite[Proposition 15.1]{reichstein2022hilberts}. We show that 
	in (small) positive characteristic, this inequality can be sharpened.
	
	\begin{theorem} \label{thm.E6} $\rd_p(W(E_6)) \leqslant 2$ if $p = 2$, $3$ or $5$.
	\end{theorem}
	
	In summary, $\rd_p(G) \leqslant d$, where the value of $d$ is given in the table below.
	
	\begin{table}[h]
		\centering \rm
		\begin{tabular}{c|ccccc}
			& \multicolumn{5}{c}{Characteristic}\\
			\hline
			$G$ & 0 & 2 & 3 & 5 & 7\\
			\hline
			$\Sym_6$ & 2 & 2 & 1 & 2 & 2\\
			$\Sym_7$ & 3 & 3 & 2 & 2 & 2\\
			$\Sym_8$ & 4 & 3 & 4 & 4 & 4 \\
			$W(E_6)$ & 3 & 2 & 2 & 2 &3\\
		\end{tabular}
	\end{table} 
 
	The remainder of this paper is structured as follows. In Section~\ref{sect.prel} we recall the definition
	of resolvent degree of a finite group and collect some of its properties for future use.  
	We believe that part (a) of Theorem~\ref{thm.main} was known classically; for lack of a reference, we include a short proof at the end of Section~\ref{sect.prel}. In Section~\ref{sect.symplectic-unitary} we prove upper bounds on the resolvent degree of finite symplectic and unitary groups.
	These upper bounds play a key role in the proofs of parts (b) and (c) of Theorem~\ref{thm.main} and
	of Theorem~\ref{thm.E6} in
	Section~\ref{sect.main-cd}. 
	Parts (d) and (e) of Theorem~\ref{thm.main} are proved in Section~\ref{sect.main-de} by a different (more geometric) argument
	inspired by our previous work on the essential dimension of symmetric groups~\cite{edens-reichstein}.
	
	\section{Preliminaries}
	\label{sect.prel}
	
	\subsection{The level of a finite field extension}
	Let $K$ be a field containing a base field $k$, and $L/K$ be a finite extension. We say that $L/K$
	descends to an intermediate field $k \subset K_0 \subset K$ if 
	$L = L_0 \otimes_{K_0} K$ for some finite extension $L_0/K_0$. The essential dimension $\ed_k(L/K)$ is then the smallest transcendence degree $\trdeg_k(K_0)$ such that $L/K$ descends to $K_0$.
	
	The level $\lev_k(L/K)$ of a finite extension $L/K$ is the smallest integer $d$ such that there exists a tower of field extensions 
	\[
	\xymatrix{ & K_m \ar@{-}[d] & \\ 
		L \ar@{-}[ur] & \vdots \ar@{-}[d] &   \\
		& K_1 \ar@{-}[d] & \\
		& K_0 \ar@{-}[luu] \ar@{=}[r] & K  }
	\]
	with $[K_{i} : K_{i - 1}] < \infty$ and $\ed_k(K_{i}/K_{i - 1}) \leqslant d$ for every $i = 1, \ldots, m$. 
	
	The resolvent degree $\rd_k(G)$ of a finite group $G$ over a field $k$ is defined as the maximal value of $\lev_k(L/K)$, where
	the maximum is taken over all fields $K$ containing $k$ and all $G$-Galois field extensions $L/K$\footnote{Note that this maximum is well defined because a $G$-Galois field extension $L/K$ with $k \subset K$ exists for any finite group $G$. Indeed, consider the regular representation $G\hookrightarrow \GL(V)$, where $V = k[G]$ is the group algebra. Now set
			$L=k(V)$ = the field of rational functions on $V$, and $K=L^G$.}.
	
	\begin{lemma} \label{lem.prel1} Let $G$ be an abstract finite group and $k$ be a field of characteristic $p \geqslant 0$.
		Then
		
		\smallskip
		(a) $\rd_k(G) = \rd_{k'}(G)$ for any field $k'$ of characteristic $p$.
		
		\smallskip
		(b)  If $H$ is a subgroup of $G$, then $\rd_k(H) \leqslant \rd_k(G)$. Moreover, if $G \neq 1$, then $\rd_k(G) \geqslant 1$.
		
		\smallskip
		(c) If $G$ is abelian, then $\rd_k(G) \leqslant 1$.
		
		\smallskip
		(d) If $1 \to A \to G \to B \to 1$ is an exact sequence, then
		$\rd_k(G) \leqslant \max \, \{ \rd_k(A), \; \rd_k(B) \}$. If additionally $A$ is a central subgroup of $G$, $B \neq 1$ and
		$\Char k\nmid |A|$, then $\rd_k(G) = \rd_k(B)$.
	\end{lemma}
	
	\begin{proof} Part (a) is Theorem 1.2 in~\cite{reichstein2022hilberts}. 
		
		(b) For the inequality $\rd_k(H) \leqslant \rd_k(G)$, see
		Lemma 3.13 in \cite{farb-wolfson} or Remark 10.5 in \cite{reichstein2022hilberts}.
		To prove the inequality $\rd_k(G) \geqslant 1$ we may replace $k$ by its algebraic closure; see part (a). 
		In the case, where $k$ is algebraically closed, $\lev_k(L/K) \geqslant 1$ for every non-trivial extension $L/K$; 
		see Lemma 4.5 in~\cite{reichstein2022hilberts}. Thus $\rd_k(G) \geqslant 1$ for every non-trivial group $G$.
		
		For (c), see Corollary~3.4 in \cite{farb-wolfson} or Example 10.6 in \cite{reichstein2022hilberts}.\footnote{Note that \cite{farb-wolfson} assumes that $\Char(k) = 0$. In~\cite{reichstein2022hilberts}, $k$ is allowed to be of arbitrary characteristic.} 
		
		For the first inequality in (d), see Theorem 3.3 in \cite{farb-wolfson} or Proposition 10.8(a) in~\cite{reichstein2022hilberts}. 
		In the case, where $A$ is central, Proposition 10.8(d) in~\cite{reichstein2022hilberts}
		tells us that 
		\begin{equation} \label{e.10.8d}
			\text{$\rd_k(G)\leqslant \max\{\rd_k(B),1\}$ and $\rd_k(B)\leqslant \max\{\rd_k(G),1\}$.}
		\end{equation}
		By our assumption $B \neq 1$ and hence, $G \neq 1$. By part (b),
		$\rd_k(B) \geqslant 1$, $\rd_k(G) \geqslant 1$. Now the inequalities \eqref{e.10.8d} translate to
		$\rd_k(G) = \rd_k(B)$.
	\end{proof}
	
	For notational simplicity, we will write $\rd_p(G)$ in place of $\rd_k(G)$, where $p = \Char(k)$ is either $0$ or a prime. This notation
	makes sense in view of Lemma~\ref{lem.prel1}(a). As we mentioned in the Introduction,
	\begin{equation} \label{e.thm1.3} \text{$\rd_0(G) \geqslant \rd_p(G)$ for any $p >0$;}
	\end{equation}
	see \cite[Theorem 1.3]{reichstein2022hilberts}.
	
	\begin{corollary} \label{cor.An-vs-Sn} $\rd_p(\Alt_n) = \rd_p(\Sym_n)$ for every $p \geqslant 0$ and every $n \geqslant 3$.
	\end{corollary}
	
	\begin{proof} By part (b) of Lemma~\ref{lem.prel1}, $1 \leqslant \rd_p(\Alt_n) \leqslant \rd_p(\Sym_n)$. It remains to prove the opposite inequality, $\rd_p(\Sym_n) \leqslant \rd_p(\Alt_n)$. Indeed,
		applying Lemma~\ref{lem.prel1}(d) to the natural exact sequence $1 \to \Alt_n \to \Sym_n \to \mathbb Z/ 2 \mathbb Z \to 1$, and remembering that
		$\rd_p(\mathbb Z/ 2 \mathbb Z) \leqslant 1$ by part (c), we obtain
		\[ \rd_p(\Sym_n) \leqslant \max \{ \rd_k(\Alt_n), \, \rd_p(\mathbb Z/ 2 \mathbb Z) \} = \max \{ \rd_p(\Alt_n), \, 1 \} = \rd_p(\Alt_n), \] 
		as desired.
	\end{proof} 
	
	\subsection{Generically free actions}
	Consider an algebraic variety $X$ equipped with the action of a finite group $G$ defined over a field $k$. 
	We will sometimes refer to such $X$ as a $G$-variety.
	We say that the $G$-action on $X$ is generically free if $G$ acts freely on a $G$-invariant dense open subvariety $U \subset X$ defined over $k$.
	In other words, we require that the stabilizer of every $\overline{k}$-point $u \in U$ should be trivial.
	Here $\overline{k}$ denotes the algebraic closure of $k$.
	
	Recall that the $G$-action on $X$ is called {\em faithful} if every non-trivial element of $G$ acts non-trivially on $X$. We record the following easy lemma for future reference.
	
	\begin{lemma} \label{lem.faithful} Let $G$ be a finite group and $X$ be a $G$-variety.
		
		\smallskip
		(a) A generically free $G$-action on $X$ is faithful. 
		
		\smallskip
		(b) If $X$ is irreducible, then the converse holds: A faithful $G$-action on $X$ is generically free.
	\end{lemma}
	
	\begin{proof} Part (a) is obvious from the definition, because $X$ has a $\overline{k}$-point with trivial stabilizer.
		For part (b), assume the contrary: a $G$-action on $X$ is not generically free.
		This means that $X$ is covered by the fixed point loci $X^g$, where $g$ ranges over the non-identity elements of $G$. Since $X$ is irreducible, we conclude that $X = X^{g_0}$ for some $1 \neq g_0 \in G$.  
		The element $g_0$ then acts trivially on $X$, and thus the $G$-action on $X$ is not faithful.
	\end{proof}
	
	Note that part (b) may fail if $X$ is allowed to be reducible.
	For example, the natural action of $\Sym_n$ on a disjoint union of $n$ points, is faithful but not generically free.
	
	\begin{lemma} \label{lem.gen-free} Let $V$ be a finite-dimensional $k$-vector space of dimension $\geqslant 1$, 
		$G$ be a finite subgroup of $\PGL(V)$ and $X$ be an irreducible $G$-invariant hypersurface of degree $d \geqslant 2$
		in $\mathbb{P}(V)$. Then the $G$-action on $X$ is generically free.
	\end{lemma}
	
	\begin{proof} We may assume without loss of generality that the base field $k$ is algebraically closed.
		Assume the contrary: the $G$-action on $X$ is not generically free.
		Then $X$ is covered by the union of the fixed point loci $\mathbb P(V)^g$, as
		$g$ ranges over $G \setminus \{ 1 \}$. Since $X$ is irreducible,
  $X \subset \mathbb P(V)^g$ for one particular $1 \neq g \in G$.
		
		Now observe that the fixed locus $\mathbb P(V)^g$ is a finite union of subvarieties of the form $\mathbb P(V_{\lambda})$,
		where $\tilde{g}$ is a preimage of $g$ in $\GL(V)$, $\lambda$ is an eigenvalue of $\tilde{g}$, and $V_{\lambda}$
		is the $\lambda$-eigenspace of $\tilde{g}$. Note that since $g \neq 1$ in $\PGL(V)$, $V_{\lambda} \subsetneq V$.
		Since $X$ is irreducible, $X \subset \mathbb P(V_{\lambda}) \subsetneq \mathbb P(V)$ for one particular $\lambda$. Since $X$ is a hypersurface, this is only possible if $X = \mathbb P(V_{\lambda})$ is a hypersurface of degree $1$. This contradicts our assumption that the degree $d$ of $X$ is $\geqslant 2$.
	\end{proof}
	
	%
	%
	\begin{lemma} \label{lem.prel2}
		Suppose $G$ is a subgroup of $\PGL_n(k)$ and there exists
		a $G$-invariant closed subvariety $X$ of $\mathbb P^n$ of degree $a$ and dimension $b \geqslant 1$ 
		(not necessarily smooth or irreducible). 
		Assume further that the $G$-action on $X$ is generically free. Then
		\[ \rd_k(G) \leqslant \, \max \{ b, \,  \rd_k(\Sym_a) \}, \] 
		where $\Sym_a$ denotes the symmetric group on $a$ letters.
	\end{lemma}
	
	\begin{proof} See Proposition 14.1(a) in~\cite{reichstein2022hilberts} or 
		Proposition 4.11 in \cite{wolfson}.
	\end{proof}

	\subsection{Proof of Theorem~\ref{thm.main}(a)} In view of Corollary~\ref{cor.An-vs-Sn}
	it suffices to show that $\rd_3(\Alt_6) \leqslant 1$. Recall that 
		$\Alt_6 \simeq \operatorname{PSL}_2(9)$; see Page $4$ of~\cite{AtlasOfFiniteGroups}. Thus there exists a faithful action of $\Alt_6$ on the projective line
	$\mathbb P^1$ defined over the field $k = \mathbb F_9$. We now apply Lemma~\ref{lem.prel2}  with $G = \Alt_6$,
	$n = 1$ and $X = \mathbb P^1$. Here we view $X$ as a closed subvariety of $\mathbb P^1$ of degree $a = 1$ and
	dimension $b = 1$. Since $X$ is irreducible, the (faithful) $\Alt_6$-action on $X$ is automatically generically free; see
	Lemma~\ref{lem.faithful}(b).
	By Lemma~\ref{lem.prel2}  we conclude that 
	\[ \text{$\rd_3(\Alt_6) =\rd_k(\Alt_6) \leqslant \max \{ 1, \rd_k(\Sym_1) \} = 1$. \quad \quad} \]
	as desired. \qed

	\section{Resolvent Degree of finite symplectic and unitary groups}
	\label{sect.symplectic-unitary}
	
	Let $n$ be a positive integer, $q = p^r$ be a prime power, and $\mathbb F_q$ be the finite field with $q$ elements.  Recall that $\operatorname{U}_n(q)$ is defined as the subgroup of elements of $\GL_n(\mathbb F_{q^2})$ which preserve the hermitian form $h$ on $\mathbb F_{q^2}^n$ defined by the formula 
	\[ h \big( (x_1, \ldots, x_n), \, (y_1, \ldots, y_n) \big) \mapsto x_1 \overline{y_1} + \ldots + x_n \overline{y_n} . \]
	Here $\mathbb F_{q^2}/\mathbb{F}_q$ is a field extension of degree $2$, and
	$x \mapsto \overline{x} = x^q$ is the unique non-trivial automorphism of $\mathbb F_{q^2}$ over $\mathbb F_q$. The group $\SU_n(q)$ is the subgroup of elements of $U_n(q)$ of determinant $1$. We also define $\PSU_n(q)=\SU_n(q)/Z(\SU_n(q))$.
	
	The group $\Sp_n(q)$ is defined in a similar manner, as the subgroup of elements
	of $\GL_n(\mathbb F_q)$ which preserve the standard symplectic form $\omega$ on $(\mathbb F_q)^n$.
	Here $n$ is assumed to be even, $n = 2m$, and 
	\[ \omega \big( (x_1, \ldots, x_{2m}), \; (y_1, \ldots, y_{2m}) \big)   = (x_1 y_2 - x_2 y_1) + \ldots + (x_{2m-1} y_{2m} - x_{2m} y_{2m-1}). \]
	Note that every hermitian form on $\bbF_{q^2}^n$ is equivalent to $h$ and every symplectic form on $\bbF_q^n$
	is equivalent to  $\omega$. The group $\PSp_n(q)$ is defined by $\PSp_n(q)=\Sp_n(q)/Z(\Sp_n(q))$.
	
	\begin{proposition} \label{prop.unitary-symplectic} Let $n \geqslant 3$ and $q = p^r$ be a prime power. Then
		
		\smallskip
		(a) $\rd_p \big( \Sp_n(q) \big) \leqslant \; \max \, \{ n - 2, \; \rd_p(\Sym_{q + 1}) \}$, and
		
		\smallskip
		(b) $\rd_p \big( \U_n(q) \big) \leqslant \; \max \, \{ n - 2, \; \rd_p(\Sym_{q + 1}) \}$.
	\end{proposition}
	
	\begin{proof} 
		We will use the following notational conventions: $x_1, \ldots, x_n$ will denote independent variables over $\bbF_q$, ${\bf x} := (x_1, \ldots, x_n)$ and ${\bf x}^q := (x_1^q, \ldots, x_n^q)$. 
		
		\smallskip
		(a) Consider the homogeneous polynomial 
		\[ f({\bf x}) = \omega \big( {\bf x} , \; {\bf x}^q) =  (x_1 x_2^q - x_2 x_1^q) + \ldots + (x_{2m-1} x_{2m}^q - x_{2m} x_{2m-1}^q) \in \bbF_q[x_1, \ldots, x_n] \]
		 of degree $q+1$. A simple application of the Jacobian criterion shows that $f({\bf x})$ cuts out a smooth
   hypersurface in $\mathbb{P}^{n-1}$. Denote this smooth hypersurface by $X$. Since $n \geqslant 3$, $X$ has to be irreducible; otherwise irreducible components of $X$ would intersect non-trivially, and their intersection point would be singular on $X$.
   
		We are now ready to complete the proof of part (a). Applying Lemma~\ref{lem.prel1}(d) to the central exact sequence 
		\begin{equation} \label{e.PSp1} 1 \to Z \to \Sp_n(q) \to G \to 1, 
		\end{equation}
		where $G$ is the image of $\Sp_{n}(q)$ under the natural projection $\GL_n(\mathbb F_q) \to \PGL_n(\mathbb F_q)$ 
		and $Z = \{ \pm 1 \}$ is the subgroup of scalar matrices in $\Sp_n(q)$, we obtain $\rd_p(\Sp_n(q)) = \rd_p(G)$.
		On the other hand, $G \subset \PGL_n(\mathbb F_q)$ acts on the irreducible hypersurface 
		$X$ of degree $q + 1 > 2$. By Lemma~\ref{lem.gen-free}, 
		the $G$-action on $X$ is generically free. 
		Applying Lemma~\ref{lem.prel2}, we obtain
		\begin{equation} \label{e.PSp3}
			\rd_p(\Sp_n(q)) = \rd_{\bbF_q} (G) \leqslant \; \max \, \{ n - 2, \; \rd_{\mathbb F_q}(\Sym_{q + 1}) \},
		\end{equation}
		as desired.
%
		
		\smallskip
		(b) We apply a similar argument to the polynomial
		\[ f({\bf x}) = h({\bf x}, {\bf x}) = x_1^{q+1} + \ldots + x_n^{q + 1} \in \mathbb F_q[x_1, \ldots, x_n] \]
		of degree $q+1$. Let $X \subset \mathbb P^{n-1}$ be the hypersurface cut out by $f({\bf x})$.
             Once again, $X$ is smooth by the Jacobian criterion, and since $n \geqslant 3$, this allows us to conclude that
             $X$ is irreducible. We claim that $f({\bf x})$ (and hence, $X$) is invariant under the natural action of $\U_n(q)$. Choose $g \in \U_n(q)$. Our goal is to prove that
		\[ \Delta({\bf x}) := f(g \cdot {\bf x}) - f({\bf x}) \in \bbF_q[x_1, \ldots, x_n] \]
		is the zero polynomial.
		Indeed, for
 every ${\bf a} = (a_1, \ldots, a_n) \in \bbF_{q^2}^n$, we have  
		\[ f(g \cdot {\bf a}) = h( g \cdot {\bf a}, \, g \cdot {\bf a}) = h( {\bf a}, \, {\bf a}) = f({\bf a}). \]
		  We conclude that $\Delta({\bf x})$ is a homogeneous polynomial of degree $q + 1$ which 
		vanishes at every $\mathbb F_{q^2}$-point of $\mathbb P^{n-1}$. By~\cite{MR98}, the minimal degree of any non-zero polynomial with this property is $q^2 + 1$. This tells us that $\Delta({\bf x})$ is the zero polynomial, thus completing the proof of the claim.
		
		To finish the proof of part (b), we argue as in part (a). Consider the central exact sequence  
			\[ 1 \to Z \to \U_n(q) \to G \to 1, \]
			where $G$ is the image of $\U_{n}(q)$ under the natural projection $\GL_n(\mathbb F_{q^2}) \to \PGL_n(\mathbb F_{q^2})$,
			and $Z$ is the abelian subgroup of scalar matrices in $\U_n(q)$. By Lemma~\ref{lem.prel1}(d),  
			\begin{equation} \label{e.Un2}
				\rd_p(\U_n(q)) = \rd_p(G) 
			\end{equation}
		On the other hand, $G \subset \PGL_n(\mathbb F_{q^2})$ acts on the irreducible hypersurface 
		$X$ of degree $q + 1 \geqslant 2$ cut out by $f({\bf x})$ in $\mathbb P^{n-1}$.
		By Lemma~\ref{lem.gen-free}, the $G$-action on $X$ is generically free. 
		Thus 
		\begin{equation} \label{e.Un3}
			\rd_p(G) = \rd_{\bbF_{q^2}}(G) \leqslant \; \max \, \{ n - 2, \; \rd_{\mathbb F_{q^2}}(\Sym_{q + 1}) \} =
			\max \, \{ n - 2, \; \rd_p(\Sym_{q + 1}), \} 
		\end{equation}
		where the first and the last equalities follow from Lemma~\ref{lem.prel1}(a), and the inequality 
		in the middle from Lemma~\ref{lem.prel2} .
		Combining~\eqref{e.Un2} and \eqref{e.Un3}, we arrive at the inequality of part (b).
	\end{proof}
	
	\section{Proof of Theorems~\ref{thm.main}(b-c) and~\ref{thm.E6} }
	\label{sect.main-cd}
	Recall that given finite groups $G,H$, central subgroups $Z_1\subset Z(G)$, $Z_2\subset Z(H)$ and an isomorphism $\varphi:Z_1\to Z_2$, we may construct the central product $G\circ_{\varphi} H$ as the quotient $(G\times H)/N$, where $N$ is the normal subgroup
	\begin{equation*}
		\{(g,h)\in G\times H:g\in Z_1,h\in Z_2,\textrm{ and }\varphi(g)h=1\}.
	\end{equation*}
	Note that the natural maps $G\to G\circ_{\varphi}H$ and $H\to G\circ_{\varphi}H$ are injective. When the subgroups $Z_1,Z_2$ and the isomorphism $\varphi$ are clear from the context, we write the central product as $G\circ H$.
	\begin{proof}[Proof of Theorems~\ref{thm.main}(b)]
		By Lemma ~\ref{lem.prel1}(a), we may work over $k=\mathbb{F}_3$. In view of Corollary~\ref{cor.An-vs-Sn}, it suffices to show that
		$\rd_{\mathbb F_3}(\Alt_7) \leqslant 2$.
		By Table 8.11 of \cite{bray_holt_roney-dougal_2013} we have an inclusion $\mathbb{Z}/4\mathbb{Z}\circ (2\cdot \Alt_7)\subset \SU_4(3)$, which induces an inclusion $2\cdot \Alt_7\subset U_4(3)$. Consequently, Proposition ~\ref{prop.unitary-symplectic}(b) implies that
		\begin{equation*}
			\rd_{\mathbb F_3}(\Alt_7) =	\rd_{\mathbb{F}_3}(2\cdot \Alt_7)\leqslant
			\rd_{\mathbb{F}_3}(U_4(3)) \leqslant \max\{4-2,\rd_{\mathbb{F}_3}(\Sym_4)\} \leqslant \max \{ 2, \rd_{\mathbb C}(\Sym_4) \} = 2.
		\end{equation*}
		Here the first equality follows from Lemma ~\ref{lem.prel1}(d), applied to the central extension $0\to \mathbb{Z}/2\mathbb{Z}\to 2\cdot \Alt_7\to \Alt_7\to 1$. The first inequality follows from Lemma~\ref{lem.prel1}(b) with $H = 2 \cdot \Alt_7$ and $G = U_4(3)$, the second inequality from Proposition ~\ref{prop.unitary-symplectic}(b), and the third inequality from~\eqref{e.thm1.3}. The equality on the right follows from 
		the fact that $\rd_{\mathbb C}(\Sym_4) = 1$; see \cite[Example 10.8]{reichstein2022hilberts} or \cite[Corollary 3.4]{farb-wolfson}.
	\end{proof}
	
	\begin{proof}[Proof of Theorems~\ref{thm.main}(c)]	
		The inequality $\rd_5(\Sym_6) \leqslant 2$ follows from~\eqref{e.thm1.3} and~\eqref{e.upper-bounds}. Moreover, $\rd_5(\Sym_6) \leqslant \rd_5(\Sym_7)$ by Lemma~\ref{lem.prel1}(b), and 
		$\rd_5(\Sym_7) = \rd_5(\Alt_7)$ by Corollary~\ref{cor.An-vs-Sn}.
		It thus suffices to show that $\rd_k(\Alt_7) \leqslant \rd_k(\Sym_6)$, where $k=\mathbb{F}_5$.
		By Table 8.6 of \cite{bray_holt_roney-dougal_2013} there is an inclusion $3\cdot \Alt_7\subset \SU_3(5)\subset U_3(5)$.
		Thus Proposition ~\ref{prop.unitary-symplectic}(b) shows that
		\[
		\rd_{\mathbb F_5}(\Alt_7) = 
		\rd_{\mathbb{F}_5}(3 \cdot \Alt_7)\leqslant \rd_{\mathbb F_3}(U_3(5)) \\
		\leqslant \max\{3-2,\rd_{\mathbb{F}_5}(\Sym_6)\}= \rd_{\mathbb{F}_5}(\Sym_6),
		\]
		as desired. Here the first equality follows from Lemma ~\ref{lem.prel1}(d), applied to the central extension $0\to \mathbb{Z}/3\mathbb{Z}\to 3\cdot \Alt_7\to \Alt_7\to 1$. The first inequality follows from Lemma~\ref{lem.prel1}(b) with
		$H = 3 \cdot \Alt_7$ and $G = U_3(5)$ and the second from  Proposition ~\ref{prop.unitary-symplectic}(b). The equality on the right
		follows the second part of Lemma~\ref{lem.prel1}(b), which tells us that $\rd_5(\Sym_6) \geqslant 1$. 
	\end{proof}

	\begin{proof}[Proof of Theorem~\ref{thm.E6}] We claim that  
		\begin{equation} \label{e.E6} \rd_p(W(E_6)) = \rd_p(\SU_4(2))
		\end{equation}
		for any $p \geqslant 0$. This follows from the exact sequence $1\to \SU_4(2)\to W(E_6)\to \mathbb{Z}/2\mathbb{Z}\to 0$; see Page $26$ of ~\cite{AtlasOfFiniteGroups}. Indeed, $\rd_p(\SU_4(2)) \leqslant \rd_p(W(E_6))$ by Lemma~\ref{lem.prel1}(b). On the other hand, 
		\[ \rd_p(W(E_6)) \leqslant	 \max \{ \rd_p(\SU_4(2)), \rd_p(\mathbb Z/ 2 \mathbb Z) \} =
		\max \{ \rd_p(\SU_4(2)), 1 \} = \rd_p(\SU_4(2)) \]
		by Lemma~\ref{lem.prel1}(b), (c) and (d).
		
		\smallskip
		$p = 2$. By Proposition~\ref{prop.unitary-symplectic}(b), 
		\[
		\rd_2(\SU_4(2)) 
		\leqslant \rd_2(\U_4(2))\leqslant \max\{4-2,\rd_2(\Sym_3) \} = 2.
		\]
		Combining this with \eqref{e.E6}, we obtain the desired inequality, $\rd_2 (W(E_6)) \leqslant 2$.
		
		\smallskip
		$p = 3$. Here we use the exceptional isomorphism $\SU_4(2)\cong \PSp_4(3)$; see Page $26$ of \cite{AtlasOfFiniteGroups}. 
		Combining \eqref{e.E6} and Proposition~\ref{prop.unitary-symplectic}(a) we obtain
		\begin{equation*}
			\rd_3(W(E_6)) =	\rd_{3}(\SU_4(2)) = \rd_3(\PSp_4(3)) = \rd_{\mathbb{F}_3}(\Sp_4(3))\leqslant \max\{4-2,\rd_{\mathbb{F}_3}(\Sym_4)\}\leqslant 2.
		\end{equation*}
		Here the equality $\rd_3(\PSp_4(3)) = \rd_3(\Sp_4(3))$ follows from Lemma~\ref{lem.prel1}(d), because
			$\Sp_4(3)$ is a central extension of $\PSp_4(3)$.
		
		\smallskip
		$p = 5$. Table 8.11 of \cite{bray_holt_roney-dougal_2013} gives an inclusion 
		$2\cdot \SU_4(2)\subset \SU_4(5)\subset \U_4(5)$. By Lemma~\ref{lem.prel1}(a), we have $\rd_5(2 \cdot \SU_4(2)) = \rd_5(\SU_4(2))$.
			Combining this with~\eqref{e.E6} with Proposition ~\ref{prop.unitary-symplectic}(b), we obtain
			\begin{equation*}
				\rd_5(W(E_6))=\rd_5(2 \cdot \SU_4(2)) \leqslant \rd_{\mathbb{F}_5}(U_4(5))\leqslant \max\{4-2,\rd_{\mathbb{F}_5}(\Sym_6)\}\leqslant \max\{2,\rd_{\mathbb{C}}(\Sym_6)\}\leqslant 2,
			\end{equation*} 	 
			where the inequality on the right follows from \eqref{e.upper-bounds}.
	\end{proof}

	\section{The varieties $Y_{123}$}
	\label{sect.123}
	
	Let $n$ be a positive integer.
	We define the closed subvariety $X_{123}$ of $\mathbb A^n$ by 
	\[ X_{123} : = \{ (x_1, \ldots , x_{n}) \in \mathbb A^{n} \; | \; s_{1} (x_1, \ldots, x_n) = s_2(x_1, \ldots, x_n) = s_3(x_1, \ldots, x_n) = 0) \}. \] 
	Here $s_j(x_1, \ldots, x_n)$ denotes the $j$th elementary symmetric polynomial 
	in $x_1, \ldots, x_n$. We denote by $Y_{123}^{(n)}\subset \mathbb{P}^{n-1}$ the projective variety cut out by the same equations. 
	Note that $X_{123}$ and $Y_{123}$ depend on $n$, which is assumed to be fixed throughout.
	The symmetric group $\Sym_n$ acts on both $X_{123}$ and $Y_{123}$ by permuting the variables.
	
	\begin{lemma}\label{lem.123}
		Let $k$ be a field of characteristic $p \geqslant 0$ and $n\geqslant 7$ be a positive integer. Then
		
		\smallskip
		(a) the symmetric group $\Sym_n$ acts transitively on the irreducible components of $X_{123}$ (respectively $Y_{123}$), 
		each of which has dimension $n - 3$ (respectively, $n-4$).
		
		\smallskip 
		(b) The projective variety $Y_{123}$ is of degree $6$ in $\mathbb P^{n-1}$. It has either one or two irreducible components. 
		If there are two components, then odd permutations in $\Sym_n$ interchange them, and even permutation
		leave each component invariant.
		
		\smallskip
		(c) The $\Sym_n$-action on $Y_{123}$ is generically free.
		
		\smallskip
		(d) If $p > 0$ and $n = p^r$ is a power of $p$, then the projective variety $Y_{123}$ is a cone over the $\Sym_n$-fixed point 
		$ (1:1:\ldots:1)$ in $\mathbb{P}^{n-1}$. 
	\end{lemma}
	
	\begin{proof} (a) The ring of invariants
		$k[X_{123}]^{\Sym_n}$ is the free polynomial $k$-algebra generated by the elements $a_4, a_5, \ldots, a_n$, where
		$a_j=s_j(x_1,\ldots,x_n) \in k[x_1, \ldots, x_n]$. Hence,
		the geometric quotient $X_{123}/\Sym_n$ is isomorphic to the affine space $\mathbb A^{n - 3}$. 
		The natural inclusion $k[X_{123}]^{\Sym_n} \hookrightarrow k[X_{123}]$ gives rise to a (finite) geometric quotient map
		$\pi \colon X_{123} \to \mathbb{A}_k^{n - 3}$. The assertions about $X_{123}$ in part (a)
		now follows from the fact that $\mathbb A^{n - 3}$ is an irreducible variety of dimension $n - 3$. The assertions about
		$Y_{123}$ follow from the fact that $X_{123}$ is the affine cone over $Y_{123}$.
		
		\smallskip
		(b) $Y_{123}$ is an $(n-4)$-dimensional closed subvariety of $\mathbb P^{n-1}$ cut out by the polynomials $s_i(x_1, \ldots, x_n)$
		of degree $i$ for $i = 1, 2, 3$. Hence, the degree $Y_{123}$ is $6$. Denote the number of irreducible components of $Y_{123}$ by $m$. The group $\Sym_n$ acts transitively on these components. Hence, $m \leqslant \deg(Y_{123}) = 6$. The $\Sym_n$-action on
		the $m$ irreducible components of $Y_{123}$ gives rise to a transitive permutation representation $\Sym_n \to \Sym_m$.
		Since $n \geqslant 7$, this permutation representation has a non-trivial kernel. An easy exercise in finite group theory shows that
		either (i) $m = 1$, i.e., $Y_{123}$ is irreducible or (ii) $m = 2$, i.e., $Y_{123}$ has two irreducible components, 
		and each component is preserved by the alternating group $\Alt_n$. 
		
		\smallskip
		(c) Assume the contrary. From the description of the irreducible components of $Y_{123}$ it follows that then the
		the action of $\Alt_n$ on each irreducible component is not generically free. By Lemma~\ref{lem.faithful}(b), this implies that
		the action of $\Alt_n$ on each irreducible component of $Y_{123}$ is not faithful. In other words, for every irreducible component of $Y_{123}$, there is a non-trivial normal subgroup $N$ of $\Alt_n$ which acts trivially on that component.
		Since $A_n$ is a simple group, $N = \Alt_n$ is the only possibility for $N$. In other words,
		$\Alt_n$ acts trivially on $Y_{123}$.
		
		This means that for every element $(y_1: \ldots : y_n) \subset Y_{123}$ and every $\sigma \in \Alt_n$, we have
		\[ (y_{\sigma(1)}, \ldots, y_{\sigma(n)} ) = \lambda(\sigma) (y_1, \ldots, y_n) \, , \]
		in $\mathbb A^n$, where $\lambda(\sigma)$ is a non-zero scalar in $\overline{k}$. It is easy to see that the map $\sigma \to \lambda(\sigma)$
		is a multiplicative character $\Alt_n \to \overline{k}^*$. Since $\Alt_n$ is a simple group, it has no non-trivial multiplicative characters.
		We conclude that $(y_{\sigma(1)}, \ldots, y_{\sigma(n)}) = (y_1, \ldots, y_n)$ in $\mathbb A^n$ for every $\sigma \in \Alt_n$. 
		Since the natural action of $\Alt_n$ on $\{ 1, \ldots, n \}$ is transitive, this is only possible if
		$y_1 = \ldots = y_n$. In other words,
		$Y_{123}$ is either empty or consists of the single point $(1: \ldots : 1)$ in $\mathbb P^{n-1}$.
		This contradicts the assertion of part (a) that $\dim(Y_{123}) = n- 4 \geqslant 3$.
		
		\smallskip 
		(d) Suppose $y = (y_1, \ldots, y_n) \in X_{123}$. We need to show that $y_{\alpha, \beta} = (\alpha y_1 + \beta,  \ldots, \alpha y_n + \beta)$ also lies in $X_{123}$ for every $\alpha, \beta \in \overline{k}$. In other words, if $s_1(y) = s_2(y) = s_3(y) = 0$, then
		$s_1(y_{\alpha, \beta}) = s_2(y_{\alpha, \beta}) = y_3(y_{\alpha, \beta}) = 0$.
		
		Indeed, $s_1(y_{\alpha, \beta}) = s_1(y) \alpha + n \beta = 0$, since we are assuming that $s_1(y) = 0$ and $n$ is a power of $p = \Char(k)$. Similarly, 
		\[ s_2(y_{\alpha, \beta}) = s_2(y) \alpha^2 + (n-1) s_1(y) \alpha \beta +  \binom{n}{2} \beta^2 = 0 \]
		in $k$ (recall that we are assuming that $n \geqslant 7$ is a prime power). Finally,
		\[ s_3(y_{\alpha, \beta}) =  s_3(y) \alpha^3 + \alpha^2 \beta (n-2) s_2(y) + \alpha \beta^2 \binom{n-1}{2} s_1(y) + \binom{n}{3} = 0 , \]
		again because $s_1(y) = s_2(y) = s_3(y)$ and $\displaystyle \binom{n}{3} = 0$ in $k$ under our assumptions on $n$ and $\Char(k)$. 
	\end{proof}
	
	\begin{remark} \label{rem.characteristic} The condition on $n$ and $\Char(k)$ in part (d) can be weakened: our proof goes through  whenever 
		\[ \binom{n}{1} = \binom{n}{2} = \binom{n}{3} = 0 \]
		in $k$. In the next section we will only need the special case, where $n = p^r \geqslant 7$, considered above. 
	\end{remark}
	
	\begin{remark} The variety $Y_{123}$ is, in fact, irreducible. 
             This can be deduced from~\cite[Corollary 3]{sdcohen-erratum} under mild assumptions on $p$
             and from
                the main result of our forthcoming paper~\cite{edens-reichstein3} in full generality.
                Lemma~\ref{lem.123}(b) only asserts that $Y_{123}$ has at most 2 irreducible components.
                We chose to go with this weaker assertion  
                because its proof is short and self-contained, and because it
                suffices for the purpose of establishing Theorem~\ref{thm.main}(d) and (e) in the next section.
	\end{remark} 
	
	\section{Proof of Theorem~\ref{thm.main}(d-e)}
	\label{sect.main-de}
	
	We continue with the notational conventions introduced in the previous section. Recall that $Y_{123}$ in the closed  subvariety of
	$\mathbb P^{n-1}$ given by \[ s_1(x_1, \ldots, x_n) = s_2(x_1, \ldots, x_n) = s_3(x_1, \ldots, x_n) = 0 , \]
	where $s_1$, $s_2$ and $s_3$ are the first three elementary symmetric polynomials.
	Lemma~\ref{lem.123}(d) asserts that when $\Char(k) = p > 0$ and $n \geqslant 7$ is a power of $p$, 
	$Y_{123}$ is a cone over the point $(1: \ldots : 1)$ in $\mathbb P^{n-1}$.
	Let us denote the ``base" of this cone by $Z_{123} \subset \mathbb P(V) \simeq \mathbb P^{n-1}$, where 
	$V = k^n/ \Delta$. Here $\Delta$ denotes the small diagonal $\operatorname{Span}_k(1, 1, \ldots, 1)$ in $k^n$. 
	In other words, points of $Z_{123}$ are in bijective 
	correspondence with lines in $\mathbb P^{n-1}$ passing through $(1:1: \ldots 1)$. 
	
	\begin{proposition} \label{prop.123} Suppose $p > 0$ and $n = p^r \geqslant 7$. 
		
		\smallskip
		(a) $Z_{123}$ is a variety of dimension $n - 5$ and degree $6$ in $\mathbb P^{n-2}$. 
		
		\smallskip
		(b) The $\Sym_n$-action on $Y_{123}$ descends to a generically free action on $Z_{123}$.
		
		\smallskip
		(c) $\rd_p(\Sym_n) \leqslant n - 5$.
	\end{proposition}
	
	The inequalities of Theorem~\ref{thm.main}(d) and (e) are immediate consequences of Proposition~\ref{prop.123}(c). Indeed, 
	setting $n = p = 7$, we obtain $\rd_7(\Sym_7) \leqslant 2$ and setting $n = 8$ and $p = 2$, we obtain $\rd_2(\Sym_8) \leqslant 3$.
	It thus remains to prove Proposition~\ref{prop.123}.
	
	\begin{proof}[Proof of Proposition~\ref{prop.123}] (a) By Lemma~\ref{lem.123}(a), $\dim(Y_{123}) = n - 4$. Since $Y_{123}$ is a cone over
		$Z_{123}$, we conclude that $\dim(Z_{123}) = \dim(Y_{123}) - 1 = n - 5$.
		
		To find the degree of $Z_{123}$, note that $Z_{123}$ 
		is isomorphic to the intersection of the cone $Y_{123}$ in $\mathbb P^{n-1}$ 
		with a hyperplane $H \simeq \mathbb P^{n-2}$ not passing through the vertex $(1: \ldots : 1)$. More precisely, the
		closed embedding $Z_{123} \hookrightarrow \mathbb P^{n-2}$ is isomorphic to the closed embedding $(Y_{123} \cap H) \hookrightarrow H$.
		It is clear from this description that the degree of $Z_{123}$ in $\mathbb P^{n-2}$ is the same as the degree of
		$Y_{123}$ in $\mathbb P^{n-1}$. By Lemma~\ref{lem.123}(b) the degree of $Y_{123}$ in $\mathbb P^{n-1}$ is $6$, and part (a) follows.
		
		As an aside, we remark that the isomorphism between $Z_{123}$ and $Y_{123} \cap H$ is not $\Sym_n$-equivariant, since $H$ may not be invariant under $\Sym_n$.
		We can still use this isomorphism, because the $\Sym_n$-action plays no role in part (a). 
		
		\smallskip
		(b) The fact that the $\Sym_n$-action on $Y_{123}$ descends to $Z_{123}$ is clear from our construction.
		To show that this action is generically free, we argue by contradiction. Assume the contrary. 
		
		\smallskip
		{\bf Claim:} $\Alt_n$ acts trivially on $Z_{123}$.
		
		\smallskip
		To prove the Claim, recall that by Lemma~\ref{lem.123}(b) either (i) $Y_{123}$ is irreducible or (ii) $Y_{123}$ has exactly two irreducible components.
		In case (i), $Z_{123}$ is also irreducible (since $Y_{123}$ is a cone over it). By Lemma~\ref{lem.faithful}(b) the $\Sym_n$-action on $Z_{123}$ is not faithful. 
		The kernel of this action is a non-trivial normal subgroup of $\Sym_n$, i.e., either the alternating group $\Alt_n$ or all of $\Sym_n$. 
		Either way, $\Alt_n$ acts trivially on $Z_{123}$. In case (ii), each irreducible component $Y_{123}'$ and $Y_{123}''$
		on $Y_{123}$ is a cone with the vertex $(1: \ldots: 1)$. Thus $Z_{123}$ has two irreducible components $Z_{123}'$ and $Z_{123}''$, 
		where $Z_{123}'$ (respectively, $Z_{123}''$) is the base of $Y_{123}'$ (respectively, of $Y_{123}''$).
		Recall from Lemma~\ref{lem.123}(b) that odd permutations in $\Sym_n$ interchange $Y_{123}'$ and $Y_{123}''$; hence, they 
		also interchange $Z_{123}'$ and $Z_{123}''$. Thus the stabilizer of any point of $Z_{123}$ away from the intersection of the two components
		lies in the alternating group $\Alt_n$. We conclude that the action of $\Alt_n$ on each component $Z_{123}'$ and $Z_{123}''$ is not generically free.
		Now the same argument as in case (i) shows that $\Alt_n$ acts trivially on both $Z_{123}'$ and $Z_{123}''$. This proves the Claim.
		
		\smallskip
		Continuing with the proof of part (b), recall that
		by Lemma~\ref{lem.123}(c), $\Sym_n$ acts generically freely on $Y_{123}$. Choose a $\overline{k}$-point $y \in Y_{123}$
		whose stabilizer in $\Sym_n$ is trivial. Note that $y \neq (1: \ldots : 1)$, because the stabilizer of $(1: \ldots : 1)$ is all of $\Sym_n$.
		Let $z$ be the point of $Z_{123}$ corresponding to the line $L$ joining $y$ to the vertex $(1: \ldots : 1)$; see the diagram below. By our assumption $\Alt_n$ fixes $z$
		and hence acts on the line $L \simeq \mathbb P^1$. Since $L$ passes through the point $y$ with trivial stabilizer in $\Alt_n$, we conclude that
		this action is faithful. On the other hand, $\Alt_n$ fixes the point $(1: \ldots :1)$ on $L$. This means that $\Alt_n$ embeds into the subgroup 
		$B \subset \Aut(\mathbb L) \simeq \PGL_2(\overline{k})$, where $B$ consists of automorphisms of $L \simeq \mathbb P^1$ fixing the point 
		$(1: \ldots : 1)$. This group is isomorphic to the subgroup of upper-triangular matrices of the form $\begin{pmatrix} \alpha & \beta \\ 0 & 1 \end{pmatrix}$  
		in $\PGL_2(\overline{k})$. 
		Note that $B$ decomposes as a semidirect product $\mathbb{G}_a(\overline{k}) \rtimes \mathbb{G}_m(\overline{k})$, where 
			$\mathbb G_a$ is the additive group of strictly upper-triangular matrices, with $\alpha = 1$, and
			$\mathbb G_m$ is the multiplicative group of diagonal matrices, with $\beta = 0$. This semidirect product decomposition shows that
			$B$ is solvable. On the other hand, $\Alt_n$ is not solvable; hence, it cannot embed into $B$.
		This contradiction completes the proof of part (b).
		
		\pgfmathsetmacro{\R}{0.05}
		\pgfmathsetmacro{\Px}{0}
		\pgfmathsetmacro{\Py}{-0.75}
		\pgfmathsetmacro{\Pz}{0.75}
		\begin{figure}[h]
			\centering
			\begin{subfigure}{0.45\textwidth}
			    \includegraphics{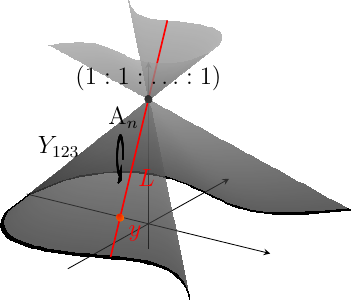}
				\caption{The variety $Y_{123}$ and the line $L\subset Y_{123}$, which has a faithful action of $\Alt_n$.}
			\end{subfigure}
			\raisebox{3.54cm}{$\Longrightarrow$}\hspace{20pt}
			\begin{subfigure}{0.45\textwidth}
				\includegraphics{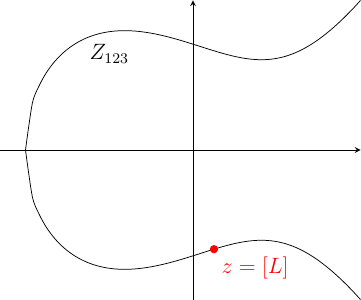}
				\caption{The variety $Z_{123}$ with the $\Alt_n$-fixed point $z=[L]\in Z_{123}$.} 
			\end{subfigure}
		\end{figure}
		\smallskip
		(c) Parts (a) and (b) allow us to apply Lemma~\ref{lem.prel2} with $G = \Sym_n$, $X = Z_{123}$, $a = 6$ and $b = n-5$. We conclude that
		\[ \rd_p(\Sym_n) \leqslant \max \{ n - 5, \rd_p(\Sym_6) \} \leqslant \max \{ n - 5, \, \rd_0(\Sym_6) \} \leqslant \max \{ n - 5, 2 \} = n-5 . \]
		Here the first inequality follows from Lemma~\ref{lem.prel2}, the second from~\eqref{e.thm1.3} and the third from~\eqref{e.upper-bounds}. The last equality follows from our assumption that $n \geqslant 7$.
	\end{proof}

	\bibliographystyle{plain}
	\bibliography{bibliography}
\end{document}